\documentclass[10pt]{amsart}

\usepackage{amsfonts,amssymb,amscd,amstext}
\usepackage[charter]{mathdesign}
\usepackage[paper size={210mm,297mm},left=34mm,right=34mm,top=43.5mm,bottom=43.5mm]{geometry}
\usepackage[colorlinks=true,linkcolor=blue,citecolor=blue]{hyperref}
\usepackage{graphicx}
\usepackage[utf8]{inputenc}
\usepackage{esint}
\usepackage{enumerate}

\renewcommand{\le}{\leqslant}
\renewcommand{\ge}{\geqslant}
\newcommand{\ptl}{\partial}

\newcommand{\rr}{{\mathbb{R}}}

\newcommand{\nn}{{\mathbb{N}}}

\newcommand{\escpr}[1]{g(#1)}

\newcommand{\Om}{\Omega}

\newcommand{\ga}{\gamma}

\newcommand{\vol}[1]{|#1|}

\newcommand{\hh}{H}

\DeclareMathOperator{\divv}{div}

\DeclareMathOperator{\intt}{int}

\DeclareMathOperator{\inj}{inj}

\newtheorem{theorem}{Theorem}[section]

\newtheorem{lemma}[theorem]{Lemma}

\theoremstyle{definition}

\theoremstyle{remark}

\newtheorem{remark}[theorem]{Remark}

\numberwithin{equation}{section}

\begin{document}

\title[The isoperimetric profile of a complete manifold]{Continuity of the isoperimetric profile of a complete Riemannian manifold under sectional curvature conditions}

\author[M.~Ritor\'e]{Manuel Ritor\'e} \address{Departamento de
Geometr\'{\i}a y Topolog\'{\i}a \\
Universidad de Granada \\ E--18071 Granada \\ Espa\~na}
\email{ritore@ugr.es}

\date{\today}

\thanks{The author has been supported by Mineco-Feder research grant MTM2013-48371-C2-1-P and Junta de Andaluc\'{\i}a grant FQM-325}

\begin{abstract}
Let $M$ be a complete Riemannian manifold possessing a strictly convex Lipschitz continuous exhaustion function. We show that the isoperimetric profile of $M$ is a continuous and non-decreasing function. Particular cases are Hadamard manifolds and complete non-compact manifolds with strictly positive sectional curvatures.
\end{abstract}

\maketitle

\thispagestyle{empty}

\section{Introduction}

Let $M$ be a complete Riemannian manifold possessing a strictly convex Lipschitz continuous exhaustion function. The aim of this paper is to show that the isoperimetric profile $I_M$ of $M$ is a continuous and non-decreasing function. In particular, Hadamard manifolds: complete, simply connected Riemannian manifolds with non-positive sectional curvatures (possibly unbounded), and complete manifolds with strictly positive sectional curvatures satisfy this assumption.

The isoperimetric profile of a Riemannian manifold is the function that assigns to a given positive volume the infimum of the perimeter of the sets of this volume. The continuity of the isoperimetric profile of a compact manifold follows from standard compactness results for sets of finite perimeter and the lower semicontinuity of perimeter \cite{MR2976521}. Alternative proofs are obtained from concavity arguments \cite[\S~7(i)]{MR875084}, \cite{bayle,MR2177105,MR1803220}, or from the metric arguments by Gallot \cite[Lemme~6.2]{MR999971}. When the ambient manifold is a non-compact homogeneous space, Hsiang showed that its isoperimetric profile is a non-decreasing and absolutely continuous function \cite[Lemma~3, Thm.~6]{MR1161609}. In Carnot groups or in cones, the existence of a one-parameter group of dilations implies that the isoperimetric profile is a concave function of the form $I(v)=C\,v^{q/(q+1)}$, where $C>0$ and $q\in\nn$, and so it is a continuous function \cite{MR1885654,MR2067135,MR2000099}. Benjamini and Cao \cite[Cor.~1]{MR1417620} proved that the isoperimetric profile of a simply connected, convex at infinity, complete surface $M^2$ satisfying $\int_MK^+\,dM<+\infty$, is an strictly increasing function. When a Riemannian manifold has compact quotient under the action of its isometry group, Morgan proved that isoperimetric regions exists for any given volume and are bounded \cite{MR2455580}, see also \cite{MR2979606}. Using the concavity arguments in \cite{bayle, MR1803220} this implies that the profile is locally the sum of a concave function and a smooth one. The author showed in \cite{MR1857855} existence of isoperimetric regions in any complete convex surface. This implies the concavity of the isoperimetric profile \cite[Cor.~4.1]{MR1857855} and hence its continuity. Nardulli \cite[Cor.~1]{MR3215337} showed the absolute continuity of the isoperimetric profile under the assumption of bounded $C^{2,\alpha}$ geometry. A manifold $N$ is of $C^{2,\alpha}$ bounded geometry if there is a lower bound on the Ricci curvature, a lower bound on the volume of geodesic balls of radius $1$, and for every diverging sequence $\{p_i\}_{i\in\nn}$, the pointed Riemannian manifolds $(M,p_i)$ subconverge in the $C^{2,\alpha}$ topology to a pointed manifold. Mu\~noz Flores and Nardulli \cite{abraham-nardulli} prove continuity of the isoperimetric profile of a complete non-compact manifold $M$ with Ricci curvature bounded below and volume of balls of radius one uniformly bounded below. Partial results for cylindrically bounded convex sets have been obtained by Ritoré and Vernadakis \cite[Prop.~4.4]{MR3385175}. For general unbounded convex bodies, the concavity of the power $I^{(n+1)/n}$ of the isoperimetric profile, where $(n+1)$ is the dimension of the convex body, has been proven recently by Leonardi, Ritoré and Vernadakis \cite{lrv}. Hass \cite{1604.02768} recently obtained examples of disconnected isoperimetric regions in Hadamard manifolds, thus showing that the corresponding isoperimetric profiles are not concave.

An example of a manifold with density with discontinuous isoperimetric profile has been described by Adams, Morgan and Nardulli \cite[Prop.~2]{cmn}. As indicated in the remark after Proposition~1 in \cite{cmn}, the authors tried to produce an example of a Riemannian manifold with discontinuous isoperimetric profile by using pieces of increasing negative curvature. By Theorem~\ref{thm:main-hadamard} in this paper, such a construction is not possible if the resulting manifold $M$ is simply connected with non-positive sectional curvatures. The first example of a complete Riemannian manifold whose isoperimetric profile is discontinuous has been recently given by Nardulli and Pansu \cite{1506.04892}.

In this paper we consider a complete Riemannian manifold $M$ of class $C^\infty$ having a strictly convex Lipschitz continuous exhaustion function $f\in C^\infty(M)$. These manifolds were considered by Greene and Wu \cite{MR0458336}, who derived interesting topological and geometric properties from the existence of such a function, e.g., such manifolds are always diffeomorphic to the Euclidean space of the same dimension \cite[Thm.~3]{MR0458336}. Complete non-compact manifolds with strictly positive sectional curvatures possess such a $C^\infty$ convex function. This follows from the existence of a continuous strictly convex function proven by Cheeger and Gromoll \cite{MR0309010} and the approximation result by $C^\infty$ functions by Greene and Wu \cite[Thm.~2]{MR0458336}. In Hadamard manifolds, the squared distance function is known to be a $C^\infty$ strictly convex exhaustion function \cite{MR823981}, although it is not (globally) lipschitz. Composing with a certain real function provides a $C^\infty$ strictly convex lipschitz continuous exhaustion function. Details are given in the proof of Theorem~\ref{thm:main-hadamard}. 

Our main result is Theorem~\ref{th:main}, where we prove the continuity of the isoperimetric profile of a complete non-compact Riemannian manifold having a strictly convex Lipschitz continuous exhaustion function $f\in C^\infty(M)$. Our strategy of proof consists on approximating the isoperimetric profile $I_M$ of $M$ by the profiles of the sublevel sets of $f$, see Lemma~\ref{lem:infI_r}. Then we show in Lemma~\ref{lem:I_r} that the strict convexity of $f$ implies that the profiles of these sublevel sets are strictly increasing. In addition, the compactness of the sublevel sets implies that these profiles are continuous. It follows that the isoperimetric profile $I_M$ is the non-increasing limit of a sequence of increasing continuous functions. Hence $I_M$ is right-continuous by Lemma~\ref{lem:right-cont}.

It only remains to show the left-continuity of $I_M$ to prove Theorem~\ref{th:main}. The main ingredient is Lemma~\ref{lem:lambda}, that plays an important role in the proof of Lemma~\ref{lem:infI_r}. In Lemma~\ref{lem:lambda} it is shown that, given a set $E\subset M$ of volume $v>0$, a bounded set $B\subset M$ of volume greater than $v$, a positive radius $r_0>0$, and a bounded set $D\subset M$ containing the tubular neighborhood of radius $r_0$ of $B$, we can always place a ball $B(x,r)$ of small radius centered at a point $x\in D$ such that $|B(x,r)\setminus E|\ge \Lambda(r)$. The expression for $\Lambda(r)$ in terms of $r$ is given in \eqref{eq:lambda} and implies that $\Lambda(r)$ approaches $0$ if and only if $r$ approaches $0$. This can be considered a refined version of Gallot's Lemme~6.2 in \cite{MR999971} (see also \cite[Lemma~2.4]{MR3215337}). Lemma~\ref{lem:lambda} will be used to add a small volume to a given set while keeping a good control on the perimeter of the resulting set. It is essential to add this small volume in a bounded subset of the manifold to use the classical comparison theorems for volume and perimeter of geodesic balls when the sectional curvatures are bounded from above and the Ricci curvature is bounded below. 

The continuity and monotonicity of the isoperimetric profiles of Hadamard manifolds, Theorem~\ref{thm:main-hadamard}, and of complete manifolds with strictly positive sectional curvatures, Theorem~\ref{thm:strictcurvature}, are corollaries of Theorem~\ref{th:main}.

A continuous monotone function can be decomposed as the sum of an absolutely continuous function and a continuous singular function (such as the Cantor function or  Minkowski's question mark function). It would be desirable to find conditions ensuring the absolute continuity of $I_M$.

An appropriate modification of the notion of convex function could make the arguments in this paper work in the case of a manifold with density.


The author wishes to thank Ana Hurtado, Gian Paolo Leonardi and C\'esar Rosales for their careful reading of the first version of this manuscript and their useful suggestions, and to the referees for their constructive comments.

\section{Preliminaries}

Given a Riemannian manifold $M$ and a measurable set $E\subset M$, we shall denote by $|E|$ its Riemannian volume. Given an open set $\Om\subset M$, the relative perimeter of $E$ in $\Om$, $P(E,\Om)$, will be defined by
\[
P(E,\Om):=\sup\bigg\{\int_E\divv X\,dM : X\in\frak{X}^1_0(\Om), ||X||_\infty\le 1\bigg\},
\]
where $\divv$ is the Riemannian divergence on $M$, $dM$ the Riemannian volume element, $\frak{X}^1_0(\Om)$ the set of $C^1$ vector fields with compact support in $\Om$, and $||\cdot||_\infty$ the $L_\infty$-norm of a vector field. The perimeter $P(E)$ of a measurable set $E\subset M$ is the relative perimeter $P(E,M)$ of $E$ in $M$.

The \emph{isoperimetric profile} of $M$ is the function $I_M:(0,|M|)\to\rr^+$ defined by
\[
I(v):=\inf\{P(E): E\subset M\ \text{measurable}, |E|=v\}.
\]
A measurable set $E\subset M$ is \emph{isoperimetric} if $P(E)=I_M(|E|)$. The isoperimetric profile function determines the isoperimetric inequality $P(F)\ge I_M(|F|)$ for any measurable set $F\subset M$, with equality if and only if $F$ is isoperimetric.

Open and closed balls of center $x\in M$ and radius $r>0$ will be denoted by $B(x,r)$ and $\overline{B}(x,r)$, respectively.

A continuous function $f:M\to\rr$ is \emph{strictly convex} if $f\circ\ga$ is strictly convex for any geodesic $\ga:I\to M$. It follows that a smooth function $f\in C^\infty(M)$ is strictly convex if and only if $(f\circ\ga)''>0$ on $I$ for any geodesic $\ga:I\to M$. A function $f:M\to\rr$ is \emph{Lipschitz continuous} if there exists $L>0$ such that $|f(p)-f(q)|\le L\,d(p,q)$ for any pair of points $p$, $q\in M$. We shall say that a continuous function $f:M\to\rr$ is an \emph{exhaustion function} if, for any $r> \inf f$, the set $C_r:=\{p\in M: f(p)\le r\}$ is a compact subset of $M$. In the sequel we shall assume the existence of a strictly convex Lipschitz continuous exhaustion function $f\in C^\infty(M)$. The following properties for $f$ and $M$ are known
\begin{enumerate}
\item $f$ has a unique minimum $x_0$, that is the only critical point of $f$.
\item The sets $\ptl C_r:=\{p\in M:f(p)=r\}$ are strictly convex hypersurfaces whenever $r>f(x_0)$. In particular, their mean curvatures are strictly positive.
\item If $f(x_0)=0$ then $B(x_0,L^{-1}r)\subset \intt(C_{r})$, and so $\overline{B}(x_0,L^{-1}r)\subset C_r$.
\item If $f(x_0)=0$ then there exists a positive constant $K$ such that $C_r\subset \overline{B}(x_0,K^{-1}r+1)$ for all $r\ge 1$.
\item $M$ is diffeomorphic to $\rr^n$.
\end{enumerate}
The uniqueness of $x_0$ follows from the arguments at the beginning of the proof of Theorem~3(a) in the paper by Greene and Wu \cite{MR0458336}. We remark that we can always normalize the function $f$, by adding a constant, so that $f(x_0)=0$. Property 2 is well-known while property 3 is obtained from the inequality $f(x)\le f(x_0)+Ld(x_0,x)$. Property 4 is a consequence of the arguments used to prove Theorem~5 in \cite{MR0458336}. We sketch here a proof for completeness: choose some $t_0\in (0,1)$ such that $\overline{B}(x_0,t_0)\subset C_1\subset C_r$. Let $K:=\inf\{f(\exp_{x_0}(t_0v)):v\in T_{x_0}M, |v|=1\}>0$. If $x\in C_r\setminus\overline{B}(x_0,t_0)$, using Hopf-Rinow's Theorem we can connect $x_0$ and $x$ by a unit-speed length-minimizing geodesic $\ga:[0,d(x_0,x)]\to M$. Since $(f\circ\ga)$ is convex we have
\[
r\ge f(x)=(f\circ\ga)(d(x_0,x))\ge (f\circ\ga)(t_0)+(f\circ\ga)'(t_0)(d(x_0,x)-t_0)\ge K(d(x_0,x)-t_0),
\]
thus implying $d(x_0,x)\le K^{-1}r+t_0<K^{-1}r+1$. If $x\in\overline{B}(x_0,t_0)$ the same inequality holds and proves that $C_r\subset \overline{B}(x_0,K^{-1}r+1)$. Finally, property 5 is proven in \cite[Theorem~3(a)]{MR0458336}.

Theorem~1(a) in the paper by Greene and Wu \cite{MR0458336} ensures the existence of a strictly convex function Lipschitz continuous exhaustion function in any complete Riemannian manifold with positive sectional curvatures. As we shall see later, such functions also exist on Hadamard manifolds, complete simply connected Riemannian manifolds with non-positive sectional curvatures.

The isoperimetric profiles $I_{r}:(0,|C_r|)\to\rr^+$ of the sublevel sets of $f$ will be defined by
\[
I_{r}(v):=\inf\{P(E):E\subset C_r\ \text{measurable}, |E|=v\}.
\]
The compactness of $C_r$ and the lower semicontinuity of perimeter imply the existence of isoperimetric regions in $C_r$ for all $v\in (0,|C_r|)$, as well as the continuity of the isoperimetric profile of $C_r$. From the definitions of $I_r$ and $I_M$ we have $I_M\le I_r\le I_s$ for all $r\ge s>f(x_0)$.

Given $\delta\in \rr$, we shall denote by $V_{\delta,n}(r)$ the volume of the geodesic ball in the $n$-dimensional complete simply connected manifold with constant sectional curvatures equal to $\delta$. When $\delta=0$, $V_{0,n}(r)=\omega_nr^n$, where $\omega_n$ is the $n$-volume of the unit ball in $\rr^n$. In case $\delta>0$, the radius $r$ will be taken smaller than $\pi/\delta^{1/2}$. In what follows, we shall take $\pi/\delta^{1/2}:=+\infty$ when $\delta\le 0$. The injectivity radius of $x_0\in M$ will be denoted by $\inj(x_0)$. If $K\subset M$, the injectivity radius of $K$ will be defined by $\inj(K)=\inf_{x\in K}\inj(x)$. When $K$ is relatively compact, $\inj(K)>0$.

The following result will play a crucial role in the sequel. Given a set $E\subset M$ of fixed volume $v$ and a small positive radius $r>0$, it allows us to place a ball $B(x,r)$ of radius $r>0$ whose center lies in a fixed bounded set $B$ (depending on the volume) so that $|B(x,r)\setminus E|\ge \Lambda(r)>0$, where $\Lambda(r)$ converges to $0$ if and only if $r$ converges to $0$.

\begin{lemma}
\label{lem:lambda}
Let $M$ be an $n$-dimensional complete non-compact Riemannian manifold, $E\subset M$ a measurable set of finite volume, $B\subset M$ a bounded measurable set such that $|B|-|E|>0$, and $\delta$ the supremum of the sectional curvatures of $M$ in $B$. Fix $r_0>0$ and choose $D\supset B$ bounded and measurable such that $d(B,\ptl D)>r_0$.  For any $0<r<\min\{r_0,\inj(B),\pi/\delta^{1/2}\}$ define
\begin{equation}
\label{eq:lambda}
\Lambda(r):=\frac{|B|-|E|}{|D|}\,V_{\delta,n}(r).
\end{equation}
Then there exists $x\in D$ such that
\[
|B(x,r)\setminus E|\ge \Lambda(r)>0.
\]
\end{lemma}

\begin{proof}
Given two measurable sets $D, F\subset M$ of finite volume, Fubini-Tonelli's Theorem applied to the function $(x,z)\in D\times M\mapsto \chi_{F\cap B(x,r)}(z)$ yields
\begin{equation*}
\int_D |F\cap B(x,r)|\,dM(x)=\int_F |B(z,r)\cap D|\,dM(z).
\end{equation*}
For $F=M\setminus E$, this formula reads
\[
\int_D|B(x,r)\setminus E|\,dM(x)=\int_{M\setminus E}|B(z,r)\cap D|\,dM(z).
\]
Since $r\le r_0$, we have $B(z,r)\cap D=B(z,r)$ for any $z\in B$, and we get the bound
\begin{align*}
\int_{M\setminus E} |B(z,r)\cap D|\,dM(z)&\ge \int_{B\setminus E} |B(z,r)\cap D|\,dM(z)
\\
&=\int_{B\setminus E} |B(z,r)|\,dM(z)\ge|B\setminus E|\,V_{\delta,n}(r)
\\
&\ge \big(|B|-|E|\big)\,\,V_{\delta,n}(r),
\end{align*}
where inequality $|B(z,r)|\ge V_{\delta,n}(r)$ follows from G\"unther-Bishop's volume comparison theorem \cite[Thm.~III.4.2]{MR2229062}. On the other hand,
\[
\int_D|B(x,r)\setminus E|\,dM(x)\le |D|\,\sup_{x\in D}|B(x,r)\setminus E|.
\]
This way we obtain
\[
\sup_{x\in D} |B(x,r)\setminus E|\ge \frac{|B|-|E|}{|D|}\,V_{\delta,n}(r),
\]
and the result follows.
\end{proof}

The following proof follows the lines in \cite[Lemma~3.1]{rv4} with the modifications imposed by the geometry of $M$.

\begin{lemma}
\label{lem:infI_r}
Let $M$ be an $n$-dimensional complete non-compact Riemannian manifold possessing  a strictly convex Lipschitz continuous exhaustion function $f\in C^\infty(M)$. Then we have
\[
I_M(v)=\inf_{r>\inf f} I_{r}(v),
\]
for every $v\in (0,|M|)$.
\end{lemma}

\begin{proof}
From the definition of $I_r$ it follows that $I_s\ge I_r\ge I_M$, for $r\ge s$, in the interval $(0,|C_s|)$. Hence $I_M \le \inf_{r>\inf f}I_r$. From now on, we assume $f$ is normalized so that $f(x_0)=0$.

To prove the opposite inequality we shall follow the arguments in \cite{MR2067135}. Fix $0<v<|M|$, and let ${\{ E_i \}}_{i\in \nn}\subset M$ be a sequence of sets of finite perimeter satisfying $\vol{E_i}=v$ and $\lim_{i\to \infty} P(E_i)=I_M(v)$.

Since $|E_i|=v<|M|$, there exists $R_i>0$ such that
\[
\vol{E_i\setminus C_{R_i}}<\frac{1}{i}.
\]
We now define a sequence of real numbers ${\{ r_i \}}_{i\in \nn}$ by taking $r_1:=R_1$ and $r_{i+1}:=\max\{r_i,R_{i+1}\}+i$. Then $\{r_i\}_{i\in\nn}$ satisfies
\[
r_{i+1}-r_i \ge i, \qquad\quad \vol{E_i\setminus C_{r_i}}<\frac{1}{i}.
\]

In case $|E_i\setminus C_{r_{i+1}}|=0$, we take a representative $G_i$ of $E_i$ contained in $C_{r_{i+1}}$ and we have
\begin{equation}
\label{eq:case1}
I_{r_{i+1}}(v)\le P(G_i)=P(E_i).
\end{equation}

In case $|E_i\setminus C_{r_{i+1}}|>0$, since $|\nabla f|\le L$, the coarea formula implies
\[
\frac{1}{L}\int_{r_i}^{r_{i+1}}{\hh}^{n-1}(E_i \cap \ptl C_t)\, dt< \vol{E_i}=v.
\]
Hence the set of $r\in [r_i,r_{i+1}]$ such that $H^{n-1}(E_i\cap\ptl C_r)\le Lv/(r_{i+1}-r_i)$ has positive measure, where $H^{n-1}$ is the $(n-1)$-dimensional Hausdorff measure in $M$. By \cite[Chap.~28, Ex.~18.3, p.~216]{MR2976521}, we can choose $ \rho(i)\in [r_i,r_{i+1}]$ in this set so that
\[
P(E_i\cap C_{\rho(i)})=P(E_i,\intt{C_{\rho(i)}})+H^{n-1}(E_i\cap C_{\rho(i)}).
\]
By the choice of $\rho(i)$ and the properties of $\{r_i\}_{i\in\nn}$ we also have
\[
{\hh}^{n-1}(E_i \cap \ptl C_{\rho(i)})\le\frac{Lv}{i}.
\]
%
%
Take now $t>0$ such that $|C_t|>v=|E_i|\ge |E_i\cap C_{\rho(i)}|$ for all $i$, and let $\delta(t)$ be the maximum of the sectional curvatures of $M$ in $C_t$. Let $v_i:=|E_i|-|E_i\cap C_{\rho(i)}|$. The sequence $\{v_i\}_{i\in\nn}$ converges to $0$ since $v_i=|E_i\setminus C_{\rho(i)}|\le |E_i\setminus C_{r_i}|<1/i$. We take $s_i$ defined by the equality
\[
v_i=\frac{|C_{t}|-|E_i\cap C_{\rho(i)}|}{|C_{2t}|}\,V_{\delta(t),n}(s_i),
\]
for $i$ large enough. From Lemma~\ref{lem:lambda} we can find, for every $i\in\nn$, a point $x_i\in C_{2t}$ such that
\[
|B(x_i,s_i)\setminus (E_i\cap C_{\rho(i)})|\ge v_i.
\]
Observe that $\lim_{i\to\infty }s_i=0$ since $\lim_{i\to\infty} |E_i\cap C_{\rho(i)}|=v<|C_t|$ and $\lim_{i\to\infty} v_i=0$. By the continuity of the functions $s\mapsto |B(x_i,s)\setminus E_i|$, we can find a sequence of radii $s_i^*\in (0,s_i]$ so that $B_i^*:=B(x_i,s_i^*)$ satisfies $|B_i^*\setminus (E_i\cap C_{\rho(i)})|=v_i$ for all $i$. For large $i$, we have the inclusions $B_i^*\subset C_{\rho(i)}$, the set $F_i:=(E_i\cap C_{\rho(i)})\cup B_i^*$ has volume $v$, and we get
\begin{equation}
\label{eq:case2}
\begin{split}
I_{r_{i+1}}(v)&\le P(F_i)\le P(E_i\cap C_{\rho(i)})+P(B_i^*)
\\
&\le  P(E_i,\intt{C_{\rho(i)}}) + {\hh}^{n-1}(E_i \cap \ptl C_{\rho(i)})+P(B_i^*)
\\
&\le P(E_i)+\frac{Lv}{i}+P(B_i^*).
\end{split}
\end{equation}
Since the balls $B_i^*$ are centered at points of the bounded subset $C_{2t}$ with radii $s_i^*$ converging to $0$, Bishop's comparison result for the area of geodesic spheres when the Ricci curvature is bounded below \cite[Thm.~III.4.3]{MR2229062} implies that $\lim_{i\to\infty} P(B_i^*)=0$. Taking limits in \eqref{eq:case1} and \eqref{eq:case2} when $i\to\infty$ we obtain $\inf_{r>\inf f} I_{r}(v)\le I_M(v)$.
\end{proof}

\begin{remark}
From the proof of Lemma~\ref{lem:infI_r} it is clear that the center of the balls $B_i^*$ must be taken in a bounded set of $M$ to have $\lim_{i\to\infty} P(B_i^*)=0$. Indeed, it is easy to produce a family of geodesic balls, each one in a hyperbolic space, with radii going to $0$ and perimeters converging to $+\infty$.
\end{remark}

The existence of a strictly convex exhaustion function on $M$ implies that the hypersurfaces $\ptl C_r=\{x\in M: f(x)=r\}$ foliate $M\setminus\{x_0\}$, where $x_0$ is the only minimum of $f$. The vector field $\nabla f/|\nabla f|$, defined on $M\setminus\{x_0\}$, is the outer unit normal to the hypersurfaces $\ptl C_r$. For any $x\in\ptl C_r$ and $e$ tangent to $\ptl C_r$ at $x$ we have
\[
g\big(\nabla_e\bigg(\frac{\nabla f}{|\nabla f|}\bigg),e\big)=\frac{1}{|\nabla f|}\,\nabla^2f(e,e)>0.
\]
Hence the hypersurfaces $\ptl C_r$ are strictly convex. The positive function $\divv(\nabla f/|\nabla f|)$ is defined on $M\setminus\{x_0\}$. Its value at $x\in\ptl C_r$ is the mean curvature of the hypersurface $\ptl C_r$ at $x$.


\begin{lemma}
\label{lem:I_r}
Let $M$ be an $n$-dimensional complete manifold $M$ possessing a strictly convex Lipschitz continuous exhaustion function $f\in C^\infty(M)$. Then the isoperimetric profile $I_r$, of the sublevel set $C_r$ is a continuous and strictly increasing function for $r>\inf f$.
\end{lemma}

\begin{proof}
Continuity follows from the compactness of $C_r$ and the lower semicontinuity of perimeter since a limit of isoperimetric regions of volumes converging to $v\in (0,|C_r|)$ is an isoperimetric region of volume $v$.

To check that $I_r$ is non-decreasing, consider an isoperimetric region $E\subset C_r$ of volume $v\in (0,\vol{C_r})$. Let $0<w<v$ and take $s\in (0,r)$ such that the set $E_s:=E\cap C_s$ has volume $w$. Choose a sequence of radii $s_i$ converging to $s$ such that $P(E\cap C_{s_i})=P(E,\intt C_{s_i})+H^{n-1}(E\cap\ptl C_{s_i})$ and
\[
\int_{E\setminus C_{s_i}} \divv X\,dM=-\int_{E\cap \ptl C_{s_i}} \escpr{X,|\nabla f|^{-1}\nabla f}\,dH^{n-1}+\int_{\ptl^*E\setminus C_{s_i}} \escpr{X,\nu_E}\,d|\ptl E|,
\]
for any vector field $X$ of class $C^1$ with compact support in an open neighborhood of $C_r\setminus \intt C_{s_i}$. In the above formula, $\ptl^*E$ is the reduced boundary of $E$ and $d|\ptl E|$ is the perimeter measure. We apply this formula to $X=\nabla f/|\nabla f|$. Since $\divv X>0$ on $M\setminus\{x_0\}$, and $\escpr{X,\nu_E}\le 1$ we have
\[
\int_{E\setminus C_{s_i}} \divv X\,dM+H^{n-1}(E\cap\ptl C_{s_i})\le P(E,M\setminus C_{s_i}).
\]
Adding $P(E,\intt C_{s_i})$ to both sides of the above inequality and estimating $P(E,\intt C_{s_i})+P(E,M\setminus C_{s_i})\le P(E)$, we get
\[
\int_{E\setminus C_{s_i}} \divv X\,dM+P(E\cap C_{s_i})\le P(E).
\]
Taking inferior limits, and using the lower semicontinuity of perimeter, we obtain
\[
P(E_s)<\int_{E\setminus C_s} \divv X\,dM+P(E_s)\le P(E),
\]
and so
\[
I_r(w)\le P(E_s)< P(E)=I_r(v).
\]
Thus $I_r$ is a strictly increasing function.
\end{proof}

\begin{remark}
We point out that only the condition $\divv X > 0$ on the set $E\setminus C_{s_i}$ has been used in the proof of Lemma~\ref{lem:I_r}. Hence the proof works if we merely assume that the level sets of the exhaustion function f have positive mean curvature and that the set of critical points of f has measure zero.
\end{remark}

The following elementary lemma will be needed to prove our main result

\begin{lemma}
\label{lem:right-cont}
Let $\{f_i\}_{i\in\nn}$ be a non-increasing $(f_i\ge f_{i+1})$ sequence of continuous non-decreasing functions defined on an open interval $I\subset\rr$. Assume the limit $f(x)=\lim_{i\to\infty} f_i(x)$ exists for every $x\in I$. Then $f$ is a right-continuous function.
\end{lemma}

\begin{remark}
The hypotheses in Lemma~\ref{lem:right-cont} do not imply the left-continuity of $f$, as shown by the following example. Taking
\[
f_i(x)=\begin{cases}
1, &0\le x,
\\
1+i\,x, & -1/i\le x\le 0,
\\
0, &x\le -1/i,
\end{cases}
\]
we immediately see that the limit of the sequence $\{f_i\}_{i\in\nn}$ is the characteristic function of the interval $[0,\infty)$, which is not left-continuous.
\end{remark}

\begin{proof}[Proof of Lemma~\ref{lem:right-cont}]
Fix $x\in I$. Let $\{x_i\}_{i\in\nn}$ be any sequence such that $x_i\ge x$. Since $f$ is a non-decreasing function, $f(x)\le f(x_i)$ for all $i$. Hence
\begin{equation}
\label{eq:right1}
f(x)\le \liminf_{i\to\infty} f(x_i).
\end{equation}

Assume now that $x=\lim_{i\to\infty} x_i$. Let us build first an auxiliary sequence $\{z_i\}_{i\in\nn}$ strictly decreasing, converging to $x$ and satisfying
\begin{equation}
\label{eq:zi}
\limsup_{i\to\infty} f(z_i)\le f(x).
\end{equation}
To this aim, starting from an arbitrary $z_1>x$ we inductively choose a point $z_i$ satisfying $x<z_i<\min\{z_{i-1},x+i^{-1}\}$ and
\[
0\le f_i(z_i)-f_i(x)\le\frac{1}{i}.
\]
This last condition follows from the continuity of $f_i$. By construction, $\{z_i\}_{i\in\nn}$ is decreasing and converges to $x$. Since $f_i\ge f$ we get
\[
f(z_i)\le f_i(z_i)\le f_i(x)+\frac{1}{i},
\] 
and taking $\limsup$ we obtain \eqref{eq:zi}. Now choose a subsequence $\{y_i\}_{i\in\nn}$ of $\{x_i\}_{i\in\nn}$ such that $\lim_{i\to\infty} f(y_i)=\limsup_{i\to\infty} f(x_i)$. Since the sequence $\{y_{i}\}_{i\in\nn}$ converges to $x$, for every $i\in\nn$, we can choose $y_{j(i)}$, with $j(i)$ increasing in $i$, such that $x\le y_{j(i)}<z_i$. As $f$ is non-decreasing,
\begin{equation}
\label{eq:right2}
\limsup_{i\to\infty} f(x_i)=\lim_{i\to\infty} f(y_i)=\lim_{i\to\infty} f(y_{j(i)})\le \limsup_{i\to\infty} f(z_{i})\le f(x)
\end{equation}
by \eqref{eq:zi}. Inequalities \eqref{eq:right1} and \eqref{eq:right2} then yield the right continuity of $f$.
\end{proof}

\section{Proof of the main result}

We give now the proof of our main result and their consequences

\begin{theorem}
\label{th:main}
Let $M$ be an $n$-dimensional complete manifold $M$ possessing a strictly convex Lipschitz continuous exhaustion function $f\in C^\infty(M)$. Then the isoperimetric profile $I_M$ of $M$ is non-decreasing and continuous.
\end{theorem}

\begin{proof}
Lemmas~\ref{lem:infI_r} and \ref{lem:I_r} imply that the profile $I_M$ is the limit of the non-increasing sequence $\{I_{r}\}_{r>\inf f}$ of continuous non-decreasing isoperimetric profiles. So $I_M$ is trivially non-decreasing and Lemma~\ref{lem:right-cont}  implies that $I_M$ is right-continuous.

To prove the left-continuity of $I_M$ at $v>0$, we take a sequence $\{v_i\}_{i\in\nn}$ such that $v_i\uparrow v$. Since $I_M$ is non-decreasing, $I_M(v_i)\le I_M(v)$. Taking limits we get $\limsup_{i\to\infty} I_M(v_i)\le I_M(v)$. To complete the proof, we shall show 
\begin{equation}
\label{eq:left-liminf}
I_M(v)\le\liminf_{i\to\infty} I_M(v_i).
\end{equation}
Consider a sequence $\{E_i\}_{i\in\nn}$ of sets satisfying $|E_i|=v_i$ and $P(E_i)\le I_M(v_i)+1/i$. By Lemma~\ref{lem:lambda}, we can find a bounded sequence $\{x_i\}_{i\in\nn}$ and a sequence of radii $\{s_i\}_{i\in\nn} $ converging to $0$ so that
\[
|B(x_i,s_i)\setminus E_i|\ge v-v_i>0.
\]
We argue now as in the final part of the proof of Lemma~\ref{lem:infI_r}: since the function $s\in [0,s_i]\mapsto |B(x_i,s)\setminus E_i|$ is continuous, there exists, for large $i$, some $s_i^*\in (0,s_i]$ such that $|B(x_i,s_i^*)\setminus E_i|=v-v_i$. Taking $F_i:=E_i\cup B(x_i,s_i^*)$ we have $|F_i|=|E_i|+|B(x_i,s_i^*)\setminus E_i|=v$, and
\[
I_M(v)\le P(F_i)\le P(E_i)+P(B(x_i,s_i^*))\le I_M(v_i)+(1/i)+P(B(x_i,s_i^*)).
\]
Taking limits we get \eqref{eq:left-liminf}.
\end{proof}

\begin{theorem}
\label{thm:main-hadamard}
The isoperimetric profile $I_M$ of a Hadamard manifold $M$ is a continuous and non-decreasing function.
\end{theorem}

\begin{proof}
We only need to construct a strictly convex Lipschitz continuous exhaustion function. Fix $x_0\in M$ and let $h=\tfrac{1}{2}\,d^2$, where $d$ be the distance function to $x_0$. Standard comparison results for the Laplacian of the squared distance function imply $\nabla^2 h\ge 1$, \cite[Chap.~3]{MR2243772}. However, $h$ is not Lipschitz continuous on $M$. We consider instead the $C^\infty$ function $m:(-1,+\infty)\to\rr^+$ defined by $m(x)=(1+x)^{1/2}$, and the composition $f=m\circ h$. Take some tangent vector $e$ of modulus $1$ at some point of $M$. Then
\begin{align*}
\nabla(m\circ h)&=(m'\circ h)\,\nabla h,
\\
\nabla^2(m\circ h)(e,e)&=(m''\circ h)\,g(\nabla h,e)^2+(m'\circ h)\,\nabla^2h(e,e).
\end{align*}
From the first formula we obtain
\[
\nabla f=\frac{d}{(1+\tfrac{1}{2}\,d^2)^{1/2}}\,\nabla d.
\]
Hence $|\nabla f|$ is uniformly bounded from above and so the function $f$ is Lipschitz continuous on $M$. From the formula for the Hessian of $(m\circ h)$ we get
\[
\nabla^2f(e,e)=-\frac{1}{4}\,\frac{1}{(1+\tfrac{1}{2}\,d^2)^{3/2}}\,g(\nabla h,e)^2+\frac{1}{2}\,\frac{1}{(1+\tfrac{1}{2}\,d^2)^{1/2}}\,\nabla^2h(e,e).
\]
By Schwarz's inequality $g(\nabla h,e)\le d$ and we have
\[
\nabla^2 f(e,e)\ge \frac{1}{2}\,\frac{1}{(1+\tfrac{1}{2}\,d^2)^{3/2}}>0.
\]
Hence $f$ is strictly convex. Since the sublevel sets of $f$ are geodesic balls, $f$ is an exhaustion function on $M$. Theorem~\ref{th:main} then implies that the isoperimetric profile of $M$ is a continuous and non-decreasing function.
\end{proof}

\begin{theorem}
\label{thm:strictcurvature}
The isoperimetric profile $I_M$ of a complete non-compact manifold $M$ with strictly positive sectional curvatures is a continuous and non-decreasing function.
\end{theorem}

\begin{proof}
The existence of a strictly convex Lipschitz continuous exhaustion function follows from  Theorem~1(a) in the paper by Greene and Wu \cite{MR0458336}. The properties of the isoperimetric profile from Theorem~\ref{th:main}.
\end{proof}


\bibliographystyle{abbrv}

\bibliography{cont-hadamard}

\end{document}